\documentclass[12pt]{amsart}
\usepackage[a4paper, left=3.1cm, right=3.1cm, bottom=3cm,top=3cm]{geometry}
\usepackage{mathscinet}
\usepackage{amsmath, amsfonts, amssymb, amsthm}
\usepackage{xcolor}
\definecolor{ocean}{rgb}{0,0.5,0.5}
\definecolor{blue}{rgb}{0.00,0.26,0.50}
\usepackage[colorlinks = true, linkcolor=ocean, citecolor=blue]{hyperref}
\newtheorem{theorem}{Theorem}[section]
\newtheorem{proposition}[theorem]{Proposition}
\newtheorem{definition}[theorem]{Definition}
\newtheorem{remark}[theorem]{Remark}

\newtheorem{question}[theorem]{Question}

\newtheorem{lemma}[theorem]{Lemma}

  \def\C{\mathbb C}
\def\diag{{\rm diag}}
\def\N{\mathbb N}
\def\C{{\mathbb C}}

\def\GL{\mathrm {GL}}

\def\<{\,<\!}
\def\>{\!>\,}

\def\C{\mathbb{C}}

\def\N{\mathbb{N}}

\begin{document}

\title[Polynomial Maps with Constants over Division Algebras]{Polynomial Maps with Constants over Division Algebras and the Generalized Kaplansky--L'vov Conjecture}
	
	
\author{Archit Gangwal, Arunava Mandal, and Somesh Verma}

\address{Department of Mathematics,
Indian Institute of Technology Roorkee,
Uttarakhand 247667, India}
\email{arcgang122@gmail.com}

\address{Department of Mathematics,
Indian Institute of Technology Roorkee,
Uttarakhand 247667, India}
\email{arunava@ma.iitr.ac.in}

\address{Department of Mathematics,
Indian Institute of Technology Roorkee,
Uttarakhand 247667, India}
\email{somesh.verma.contact@gmail.com}

\begin{abstract}
The Kaplansky--L'vov conjecture asserts that the image of a multilinear polynomial map on a full matrix algebra over a field is always a vector space. Although the conjecture remains open in general, substantial progress has been made for $2\times 2$ and $3\times 3$ matrix algebras over various fields. Recently, Panja, Saini, and Singh formulated a generalized Kaplansky–L'vov conjecture for polynomial maps with matrix coefficients over algebraically closed fields and verified it for $2\times 2$ matrices.
In this work, we investigate an analogous problem for polynomial maps with constant matrix coefficients over an infinite division algebra. Specifically, we consider polynomials in the free algebra $M_2(\mathbb D)\langle x_1,\ldots,x_m\rangle$ of the form
$\omega = A_1x_1^{k_1}+\cdots+A_mx_m^{k_m},$ where the $A_1,\ldots, A_m\in M_2(\mathbb D)$ are fixed matrices, $\mathbb D$ is an infinite division algebra, and $k_1,\ldots, k_m$ are positive integers. We prove that the corresponding generalized Kaplansky--L'vov conjecture holds for $2\times 2$ matrices over $\mathbb R$ and the quaternion division algebra $\mathbb H$. We also investigate the surjectivity of these polynomial maps. This can be viewed as a generalized Waring problem for matrix algebras.
\end{abstract}
	
\maketitle
  \subjclass{\it 2020 Mathematics Subject Classification:} Primary 16K20, 16S50; Secondary 15A24, 11P05.
  \keywords{\bf Keywords:} Polynomial Maps with Constant,  Division algebra.
  
\setcounter{tocdepth}{1}
\tableofcontents
    
\section{Introduction}
A central theme in the study of polynomial maps is understanding the set of values they attain. This seemingly simple question lies at the crossroads of several areas of mathematics, including algebraic geometry, the theory of polynomial identities, and Waring-type problems. The study of polynomial images on algebras has received significant attention in recent years \cite{ASL12, KG13, ASL16, NKE18, ASL20, M20, GC22, GC23, MP23}, and see the references cited therein. Much of the related work is motivated by the well-known Kaplansky-L'vov conjecture \cite{I57}, which says that the image of a multilinear polynomial on the full matrix algebra $M_n(\mathbb K)$ over a field $\mathbb K$ is a vector space, as well as by the broader expectation that phenomena in simple algebras should parallel those in (finite) simple groups.
The analogous problem in group theory, concerning images of word maps, has also been extensively studied in recent times, though it dates back to the work of Borel \cite{B83}, who proved that for any simple algebraic group $G$ over algebraically closed field, the image, denoted by $w(G)$, of any nontrivial word $w:G^n\to G$ is dense in $G$, in particular $G=w(G)^2$. For connected simple Lie group $G$, it is classified when the images of the $k$-th power map $P_k$, given by $g\mapsto g^k$ for all $g\in G$, is dense \cite{BM18}, \cite{Ma18}, and see \cite{M21}, \cite{MS21}, \cite{MP26} for general Lie groups. W\"ustner proved that for a connected Lie group $G$ and an integer $k\geq 2$, every element $g\in G$ can be written as $g=a^kb^k$ for some $a, b\in G$, in particular, $G=P_k(G)^2$ \cite{W03}. 
Larsen, Shalev, and Tiep showed that every element of a finite simple group is a product of two values of $w$ \cite{LST11}. 
Following Larsen's suggestion, analogous questions were studied for matrix algebras. It was shown that, for sufficiently large $q$, every matrix in $M_n(\mathbb F_q)$ is a sum of two $k$-th powers, equivalently, the map $x^k+y^k$ is surjective on $M_n(\mathbb F_q)$ \cite{KA25}. This was later extended to polynomial maps of the form $\sum_{i=1}^m \delta_i x_i^{k_i}$ over algebras $M_n(\mathbb D)$, where $\mathbb{D}=\mathbb{C}, \mathbb{R}$, $\mathbb {F}_q$ and $\mathbb{H}$, and $\delta_i\in\mathbb D$ are fixed elements \cite{SPA25}.
These results belong to the broader class of Waring-type problems, which ask whether every element of the algebra generated by the image of a polynomial map can be expressed as a bounded sum of values of that map. Such problems have been investigated in a variety of algebraic settings, including central simple algebras, matrix algebras over rings \cite{KG13}, and upper triangular matrix algebras \cite{SS23}. A parallel conjecture to Kaplansky–L'vov  conjecture for upper triangular matrix algebras, known as the Fagundes--Mello conjecture, has also been resolved in \cite{GT22} and \cite{LW22}.
On the other hand, the Kaplansky–L'vov  conjecture remains open in its general form,  with complete answers currently restricted to small matrix sizes. Kanel--Belov, Malev, and Rowen settled the case $n = 2,$ with mild hypotheses on the field \cite{ASL12}, and the case $n = 3$ in subsequent work \cite{ASL16}. 
An overview of the conjecture is given in the survey of Kanel--Belov, Malev, Rowen, and Yavich \cite{ASL20}.
More recently, Panja, Saini, and Singh proposed a generalized Kaplansky-L'vov conjecture for polynomial maps with matrix coefficients over algebraically closed fields and established it for $2\times2$ matrices \cite{PSS25}, and some partial results by Saini and Singh for $3\times 3$ matrices, and some remarks for general $n\times n$ matrices \cite{SS26}.
This naturally leads to the following question.

\begin{question}\label{Q1}
Let $\mathbb D$ be a division algebra over an infinite field, and let $A,B\in M_n(\mathbb D)$ be nonzero matrices. Determine necessary and sufficient conditions under which the polynomial map $\omega(x,y)=Ax^k+By^\ell$
is surjective on $M_n(\mathbb D)$, and characterize when the image of $\omega$ is a vector space.
\end{question}

Let $\mathbb D$ be a division algebra with center $\mathbb K$, let $\mathcal A$ be a $\mathbb K$-algebra, and let
$F_n=\mathbb K\langle x_1,\dots,x_n\rangle$ be the free associative algebra on $n$ variables. Every element $w\in F_n$ induces, by evaluation, a polynomial map
$w:\mathcal A^n\to\mathcal A.$
More generally, elements of the free product $\mathcal A*F_n=\mathcal A\langle x_1,\cdots,x_n\rangle$, called \emph{generalized polynomials}, induce maps $\widetilde{\omega}:\mathcal A^n\to\mathcal A$ through evaluation (cf. \cite{PSS25}). Following the terminology of word maps with constants in group theory \cite{NKE18}, we call these maps \emph{polynomial maps with constants}. 
Throughout this note, we are primarily interested in the case $\mathcal A=M_n(\mathbb D)$. Typical examples of polynomial maps with constants include
$\omega(x_1,\ldots,x_m)
    =A_1x_1^{k_1}+\cdots+A_mx_m^{k_m},
\; A_i\in M_n(\mathbb D),$
and
$\omega(x,y)=Axy-Byx,
\; A,B\in M_n(\mathbb D).$
The first may be viewed as a generalized diagonal map, while the second is a generalized commutator-type map.
In this note, we answer Question \ref{Q1} for $2\times2$ matrices over $\mathbb{R}$ and the quaternion algebra $\mathbb{H}$. 

To state our result, we first note the following.
Let $\mathbb D\in\{\mathbb R,\mathbb H\}$ and  $A\in M_2(\mathbb D)$ be a non-zero non-invertible
matrix. Then, $A$ takes one of
the following two forms:
\begin{equation}\label{eq:intro-ratio-form}
   A \;=\;
   \begin{pmatrix}
      a_1 & a_2 \\ ta_1 & ta_2
   \end{pmatrix}
   \qquad\text{or}\qquad
   A \;=\;
   \begin{pmatrix}
      0 & 0 \\ a_1 & a_2
   \end{pmatrix},
\end{equation}
where $a_1,a_2\in\mathbb D$ are not both zero and $t\in\mathbb D$. See Section \ref{section-preliminaries} for a proof when $\mathbb D=\mathbb H$.

\begin{definition}[Row dependency ratio]\label{def:intro-rdr}
Let $\mathbb D\in\{\mathbb C, \mathbb R,\mathbb H\}$, let $A\in M_2(\mathbb D)$ be a nonzero non-invertible
matrix written in the form~\eqref{eq:intro-ratio-form}. We say, the
\emph{row dependency ratio} of $A$ denoted by $\rho(A)$, is the element
of $\mathbb D\cup\{\infty\}$, defined by
\[
   \rho(A) \;=\;
   \begin{cases}
      t & \text{if $A$ has the first form in~\eqref{eq:intro-ratio-form},}\\
      \infty & \text{if $A$ has the second form in~\eqref{eq:intro-ratio-form}
                       (i.e. has zero first row).}
   \end{cases}
\]
\end{definition}

\begin{theorem}\label{thm:image-phi}
Let $\mathbb D\in\{\mathbb R, \mathbb C, \mathbb H\}$, let $A\in M_2({\mathbb D})$ be a nonzero non-invertible
matrix, and let $k\in\mathbb N$. Let $\phi_{A,k}:M_2(\mathbb D)\to M_2(\mathbb D)$ be a map given by $X\to AX^k$. 
Then
${\rm Im}(\phi_{A,k})=
    \{\,M\in M_2(\mathbb{D})\setminus{\rm GL}_2(\mathbb D) : \rho(M)=\rho(A)\,\}\cup\{\theta\},$ where $\theta$ is the zero matrix in $M_2(\mathbb D)$.
\end{theorem}

\begin{theorem}\label{thm-surjectivity}
Let $\mathbb{D}\in\{\mathbb{R},\mathbb{H}\}$, and let
$A,B\in M_2(\mathbb{D})$
be nonzero matrices with integers $k_1,k_2\ge 1$. Let
$\omega:M_2(\mathbb{D})^2\to M_2(\mathbb{D}),$ be the polynomial map with constant $A$ and $B$, given by $\omega(X,Y)=AX^{k_1}+BY^{k_2}.$
Then $\omega$ is surjective if and only if one of the following conditions holds:
\begin{enumerate}
    \item at least one of $A$ or $B$ is invertible;
    \item $A$ and $B$ are both non-invertible and
    $\rho(A)\neq \rho(B),$ where $\rho(\cdot)$ denotes the row dependency ratio.
\end{enumerate}
Equivalently, ${\rm Im}(w)\neq M_2(\mathbb D)$ if and only if both $A$ and $B$ are non-invertible, 
$\rho(A)=\rho(B),$ and
in this case $\operatorname{Im}(\omega)=
\{\,M\in M_2(\mathbb{D})\setminus{{\rm GL}_2(\mathbb D)} : \rho(M)=\rho(A)\,\}\cup\{\theta\},$ where $\theta$ is the zero matrix in $M_2(\mathbb D)$.
\end{theorem}

\begin{remark}
\begin{enumerate}
 \item In \cite{PSS25}, Panja, Saini, and Singh proved that the analogous result over algebraically closed fields. Their characterization of surjectivity is formulated differently from Theorem~\ref{thm-surjectivity}: they showed that the map $w$ is surjective if and only if $A$ and $B$ can be simultaneously conjugated to matrices with no common zero row, and they provided a complete description of all possible images (cf. Theorem A of \cite{PSS25}). However, it is easy to observe that our formulation is equivalent to their formulation. Thus, Theorem~\ref{thm-surjectivity} may be viewed as a reformulation of their surjectivity criterion.

 \item A key ingredient in their approach to proving Theorem A of \cite{PSS25} is the use of conjugation invariance. Since conjugation by an element of $\mathrm{GL}_2(\mathbb K)$ is an automorphism of $M_2(\mathbb K)$, properties such as surjectivity and the vector space structure of the image are preserved under simultaneous conjugation of the coefficient matrices. Thus, after conjugating, one may assume that one of the coefficients, say $A$, is in Jordan canonical form, since they worked over algebraically closed fields. The centralizer $C_{\mathrm{GL}_2(\mathbb K)}(A)$ then acts on the second coefficient $B$ while fixing $A$, and the image of the associated polynomial map is preserved up to conjugation under this action. Consequently, properties such as surjectivity and the vector space structure of the image depend only on the $C_{\mathrm{GL}_2(\mathbb K)}(A)$-orbit of $B$. Thus the problem reduces to the classification of $C_{\mathrm{GL}_2(\mathbb K)}(A)$-orbits on $M_2(\mathbb K)$. 
 In contrast, in our proof of Theorem \ref{thm-surjectivity}, we follow a completely different strategy.
 Since we are working over reals, we will not get Jordan canonical forms for a general matrix. Instead, 
 we use a trace--determinant criterion (cf. Lemma \ref{lem:image-power-map-real}) for the existence of $k$-th roots of matrices. Moreover, our proof provides an alternative proof when $\mathbb D=\mathbb C$.
  \end{enumerate}
\end{remark}

\section{Preliminaries}\label{section-preliminaries}
In this section, we recall some useful facts and prove some elementary results that we will use crucially to prove our main results. We first recall Proposition 2.1 of \cite{SPA25} to prove Lemma \ref{lem:image-power-map-real}.
 
\begin{proposition}\label{prop-extension-field}
Let $\mathbb K$ be a field and $A \in M_n(\mathbb K)$ with a separable characteristic polynomial.
Then the equation $f(X_1,X_2,\ldots,X_m)=A$ in $M_n(\mathbb K)$ has a solution if the equation
$f(X_1,X_2,\ldots,X_m)=J_{\alpha,l}$
has a solution in $M_l(\mathbb K(\alpha))$ for all $\alpha$, which are eigenvalues of
$A$ over $\overline{\mathbb K}$.
\end{proposition}

\begin{lemma}[Trace-Determinant criteria]\label{lem:image-power-map-real}
Let $k \geq 2$, and let $P_k:M_2( \mathbb{R}) \to M_2( \mathbb{R})$ denote the $k$-th power map given by $P_k(X)=X^k$. Then the following statements hold:

\begin{enumerate}
    \item If $k$ is odd, then
    $\operatorname{Im}(P_k)=\{A\in M_2( \mathbb{R}) \mid A \text{ is not nilpotent}\}.$

    \item If $k$ is even, then every matrix $A\in M_2( \mathbb{R})$ satisfying any one of the following conditions belongs to $\operatorname{Im}(P_k)$:
    \begin{enumerate}
        \item $A$ has non-real eigenvalues of the form $\alpha=a+ib$, where $b\neq 0$;
        
        \item $\operatorname{tr}(A)\geq 0$ and $\operatorname{det}(A)\geq 0$, with $\operatorname{tr}(A)$ and $\operatorname{det}(A)$ not simultaneously zero.
        \end{enumerate}
        \item If $A=\lambda I$ for $\lambda\in\mathbb R$, where $I$ is the identity matrix in $M_2(\mathbb R)$, then $A\in {\rm Im}(P_k)$ for any $k\geq 2$.
\end{enumerate}
\end{lemma}

\begin{proof}
   Suppose that $A\in M_2(\mathbb R)$ has eigenvalues $\alpha\in\mathbb C$ of the form $\alpha=a+ib$, $b\neq 0$ and $\bar\alpha=a-ib$. Then $\alpha$ and $\bar\alpha$ have $k$-th root in $\mathbb C$, and hence by Proposition \ref{prop-extension-field}, $X^k=A$ has a solution in $M_2(\mathbb R)$, and hence $A\in \operatorname{Im}(P_k)$. One can also directly prove this without using Proposition \ref{prop-extension-field}.

 $(1)$ Suppose that $k$ is odd. By the above observation, we may assume all eigenvalues of $A$ are real. If the eigenvalues are distinct, then up to conjugation by an invertible matrix, we may assume $\diag(a_{11}, a_{22})$, and for each $a_{ii}$ has a real $k$-th root, and hence there exists a real matrix $X$ such that $X^k=A$. 
If the eigenvalues are the same, then up to conjugation by an invertible matrix, it is an upper-triangular matrix with diagonal entries $a$, or it is a diagonal matrix with diagonal entries $a$. If $A$ is diagonal, then by the above observation, we are done. 
Suppose that
$
A=\begin{pmatrix} a & 1 \\ 0 & a \end{pmatrix}=aI+N,
$
where
$
N=\begin{pmatrix}0&1\\0&0\end{pmatrix}
$
satisfies $N^2=0$. If $a\neq 0$ and $k$ is odd, then $a$ has a unique real $k$-th root $x\neq 0$ such that $x^k=a$. Writing $X=xI+yN$, we have
$
X^k=(xI+yN)^k=x^kI+kx^{k-1}yN,
$
since $N^2=0$. Thus $X^k=A$ if and only if $x^k=a$ and $kx^{k-1}y=1$, so
$
y=\frac{1}{k}x^{-(k-1)}.
$
Hence
$
X=\begin{pmatrix}
x & \frac{1}{k}x^{-(k-1)}\\
0 & x
\end{pmatrix}
$
is a real $k$-th root of $A$. On the other hand, if $a=0$ and $k>1$, then $A=N$ is nilpotent, and no such root exists since every $2\times2$ nilpotent matrix $X$ satisfies $X^2=0$, implying $X^k=0\neq N$.
 This proves $(1)$.

$(2)$ If $(a)$ holds, then by above observation, $A\in {\rm Im}(P_k)$. Suppose $(b)$ holds. 
 Let $\lambda_1,\lambda_2$ be the eigenvalues of $A$. Since $\lambda_1+\lambda_2=\operatorname{tr}(A)\geq0$ and $\lambda_1\lambda_2=\det(A)\geq0$, the eigenvalues are either both non-negative real numbers or a pair of (non-real) complex conjugates with non-negative real parts. If they are real, then $A$ is similar over $\mathbb R$ either to $\begin{pmatrix}\lambda_1&0\\0&\lambda_2\end{pmatrix}$ or to $\begin{pmatrix}\lambda&1\\0&\lambda\end{pmatrix}$ with $\lambda\geq0$. In the diagonal case, since $\lambda_1, \lambda_2\geq 0$, both are not zero simultaneously, it has $k$-root, say $\lambda_i^{1/k}$ for $i=1,2$, we have $\begin{pmatrix}\lambda_1^{1/k}&0\\0&\lambda_2^{1/k}\end{pmatrix}^k=\begin{pmatrix}\lambda_1&0\\0&\lambda_2\end{pmatrix}$. In the second case, necessarily $\lambda>0$ since $(\operatorname{tr}(A),\det(A))\neq(0,0)$, and if $B=\begin{pmatrix}\lambda^{1/k}&\frac{1}{k\lambda^{(k-1)/k}}\\0&\lambda^{1/k}\end{pmatrix}$, then $B^k=\begin{pmatrix}\lambda&1\\0&\lambda\end{pmatrix}$. If the eigenvalues are non-real, then by the observation in the first paragraph of the proof, we say that $A\in {\rm Im}(P_k).$ 
 
 
 (3) Suppose $A=\lambda I_2$ with $\lambda<0$. Let $a=(-\lambda)^{1/k}>0$ and consider the matrix
$B=a\begin{pmatrix}
\cos(\pi/k) & -\sin(\pi/k)\\
\sin(\pi/k) & \cos(\pi/k)
\end{pmatrix}.
$
Then,
$
B^k=a^k(-I_2)=-a^kI_2=\lambda I_2=A.
$
Therefore, $A$ admits a real $k$-th root. The rest of the part is easy to see, and hence we omit the details.
 \end{proof}
 \begin{remark}
    Lemma \ref{lem:image-power-map-real}(1) holds for $P_k:M_n(\mathbb R)\to M_n(\mathbb R)$ for $k$ odd. 
 \end{remark}


\subsection{Dieudonne determinant in division algebra}
Let $\mathbb D$ be a division algebra with center $\mathbb K$. For matrices over a commutative unital ring $R$, the determinant is the unique alternating $R$-multilinear map
$\det : M_n(R)\to R$
satisfying $\det(I_n)=1_R$. In particular, for a $2\times 2$ matrix
$A=\begin{pmatrix} a & b \\ c & d \end{pmatrix}$, the determinant is $\det(A)=ad-bc.$
For matrices over a noncommutative division algebra $\mathbb D$, the classical determinant is replaced by the \emph{Dieudonne determinant}, introduced by Dieudonne in 1943 \cite{D43}. The function requires the multiplicative property;
$\operatorname{Det}(AB)=\operatorname{Det}(A)\operatorname{Det}(B),
\; A,B\in M_n(\mathbb D).$ By Theorem 1.6 in \cite{B68}, this function
vanishes for all singular matrices in $M_n(\mathbb D)$. Then, it suffices to define the determinant
over the group of all non-singular matrices ${\rm GL}_n(\mathbb D)$.
By Definition~3 (page~135) of~\cite{Skewfield83}, the Dieudonne determinant is the group homomorphism
$
\operatorname{Det}:{\rm GL}_n(\mathbb D)\to \mathbb D^\times/[\mathbb D^\times,\mathbb D^\times],
$
induced by the map 
${\rm\delta et}:M_n(\mathbb D)\to\mathbb D$ defined by

$${\rm\delta et}(A)=
\begin{cases}
 0 , & A\in M_n(\mathbb D)\setminus{{\rm GL}_n(\mathbb D)}, \\[2mm]
 \operatorname{sgn}(\pi)\prod_{i=1}^n u_i, & A\in {\rm GL}_n(\mathbb D).
\end{cases}$$
where $u_1,\ldots,u_n$ are the nonzero diagonal entries of the upper triangular matrix in the Bruhat normal form of $A$, and $\pi$ is the corresponding permutation (Definition 1, page 133 in \cite{Skewfield83}). Thus,
$
\operatorname{Det}(A)=[{\rm\delta et}(A)]
$
for $A\in{\rm GL}_n(\mathbb D)$, and $\operatorname{Det}(A)=[0]$ if $A$ is singular.
For a $2\times 2$ matrix
$A=
\begin{pmatrix} a & b \\ c & d \end{pmatrix}
$
the Dieudonne determinant 
$$
\operatorname{Det}(A)=
\begin{cases}
\left[-bc\right], & a=0, \\[2mm]
\left[ad-aca^{-1}b\right], & a\neq 0.
\end{cases}
$$
where $[x]$ denotes the class of $x\in \mathbb D^\times$ in the abelianized multiplicative group $\mathbb D^\times/[\mathbb D^\times,\mathbb D^\times]$ (cf. \cite{products-of-hermitian-matrices25}, p-533).

\subsection{The quaternion algebra}\label{sec:prelim-H}
The quaternion algebra $\mathbb{H}$ is the unique four-dimensional non-commutative division algebra over $\mathbb{R}$ with basis $\{1,i,j,k\}$ satisfying $i^2=j^2=k^2=ijk=-1$, which imply $ij=k=-ji$, $jk=i=-kj$, and $ki=j=-ik$. Every quaternion $q\in\mathbb{H}$ can be written uniquely as $q=a+bi+cj+dk$ for $a,b,c,d\in\mathbb{R}$. The conjugate and norm of $q$ are defined by $\bar q=a-bi-cj-dk$ and $|q|=\sqrt{q\bar q}=\sqrt{a^2+b^2+c^2+d^2}$, respectively. The norm is multiplicative, i.e. $|pq|=|p||q|$ for all $p,q\in\mathbb{H}$, and every nonzero quaternion is invertible with inverse $q^{-1}=\bar q/|q|^2$. Thus $\mathbb{H}$ forms a division algebra, although multiplication is generally non-commutative. 

It is important to note that the classical determinant formula does not extend to $M_2(\mathbb{H})$. For a matrix 
$A=\begin{pmatrix} a & b \\ c & d \end{pmatrix}\in M_2(\mathbb{H})$, the four naive analogues of the determinant are
$ad-cb,\; ad-bc,\; da-cb,\; da-bc.$
However, none of these expressions, individually or collectively, characterizes invertibility. For instance, the unitary matrices
\[
A=\frac{1}{\sqrt{2}}\begin{pmatrix}1&i\\ j&k\end{pmatrix},
\qquad
B=\frac{1}{\sqrt{2}}\begin{pmatrix}i&j\\ j&i\end{pmatrix}
\]
are both invertible, yet for $A$ exactly two of the above expressions vanish, while for $B$ all four vanish. We use
the Dieudonn\'e determinant, as defined in the earlier subsection. For the Hamilton quaternion division algebra $\mathbb{H}$, the multiplicative group of nonzero quaternions is $\mathbb{H}^{\times}$. Every nonzero quaternion can be written in the form $q=ru$, where $r\in \mathbb{R}_{>0}$ and $u$ is a unit quaternion. Hence
$\mathbb{H}^{\times}\cong \mathbb{R}_{> 0}\times \mathrm{SU}(2).$
Since $\mathrm{SU}(2)$ is perfect, its commutator subgroup coincides with itself. Therefore,
$
[\mathbb{H}^{\times},\mathbb{H}^{\times}]
=
\{q\in\mathbb{H}^{\times}:N(q)=1\},
$
where $N(q)$ denotes the quaternion norm. Consequently,
$
\mathbb{H}^{\times}/[\mathbb{H}^{\times},\mathbb{H}^{\times}]
\cong
\mathbb{R}_{> 0}.
$
Thus, the Dieudonné determinant of a matrix over $\mathbb{H}$ may be regarded as taking values in $\mathbb{R}_{> 0}$ on invertible matrices.

\begin{lemma}\label{lem:non-invertible-form}
Let $A\in M_2(\mathbb{H})$ be a non-zero non-invertible matrix. Then $A$ has one of the forms
$A=
\begin{pmatrix}
a_1 & a_2\\
ta_1 & ta_2
\end{pmatrix}\;\text{or}\;
A=
\begin{pmatrix}
0 & 0\\
a_1 & a_2
\end{pmatrix},$
for some $a_1,\;a_2,\;t\in\mathbb{H}$ with $a_1$ and $a_2$ not both zero.
\end{lemma}

\begin{proof}
Let $A=\begin{pmatrix}a&b\\ c&d\end{pmatrix}\in M_2(\mathbb H)$
be non-zero and non-invertible. If $a\neq 0$ and $d-ca^{-1}b\neq 0$, then both $a$ and $d-ca^{-1}b$ are invertible, and $A=
\begin{pmatrix}
1&0\\
ca^{-1}&1
\end{pmatrix}
\begin{pmatrix}
a&b\\
0&d-ca^{-1}b
\end{pmatrix},
$
expresses $A$ as a product of two invertible matrices, a contradiction. Hence
$d=ca^{-1}b.$
If $b\neq 0$, setting $t=ca^{-1}$ gives $c=ta$ and $d=tb$, so
$A=\begin{pmatrix}a&b\\ ta&tb\end{pmatrix}.$
If $b=0$, then $d=0$, and writing $t=ca^{-1}$ yields
$A=\begin{pmatrix}a&0\\ ta&0\end{pmatrix}.$
Now suppose $a=0$. If $b\neq 0$ and $c\neq 0$, then
$A=
\begin{pmatrix}
b&0\\
d&c
\end{pmatrix}
\begin{pmatrix}
0&1\\
1&0
\end{pmatrix},
$
where both factors are invertible, contradicting the singularity of $A$. Thus either $b=0$ or $c=0$. If $b=0$, then
$A=\begin{pmatrix}0&0\\ c&d\end{pmatrix},$
which is of the second form. If $c=0$, writing $t=db^{-1}$ gives
$A=\begin{pmatrix}0&b\\ 0&tb\end{pmatrix},$
which is of the first form. Since $A\neq0$, the entries in the nonzero row are not both zero.
\end{proof}

We now recall some facts from \cite{W55}, \cite{SPA25} and \cite{L14}.

\begin{lemma}\cite{SPA25}\label{lem:roots-H}
For every $q\in\mathbb H$ and every $k\in\N$ there exists $w\in\mathbb H$ with $w^k=q$.
\end{lemma}

\begin{lemma}\protect{\cite[Theorem 1]{W55}}, \cite[Lemma 3]{L01}
\label{lem:jordan-H}
Every matrix $A\in M_n({\mathbb H})$ is similar, via a transformation in
$\GL_n(\mathbb H)$, to a complex Jordan canonical form whose diagonal entries are
of the form $a+bi$ with $b\geq 0$. That is, there exists
$\lambda_1,\dots,\lambda_r\in\C\subset\mathbb H$ (with non-negative imaginary
part) and integers $n_1,\dots,n_r$ such that
$A$ is similar to a matrix of the form $J_{\lambda_1,n_1}\oplus J_{\lambda_2,n_2}\oplus\cdots
   \oplus J_{\lambda_r,n_r},$ where $\lambda_k=a_k+ib_k$ being the right eigen values of $A$.
\end{lemma}

\begin{lemma}\protect{\cite[Section 5]{SPA25}}\label{lem:invertible-kth-root}
Let $\mathbb D\in\{\mathbb C,\mathbb H\}$, let $A\in M_2(\mathbb D)$ be invertible, and let $k\ge 1$ be an integer. Then there exists $B\in M_2(\mathbb D)$ such that
$B^k=A.$
\end{lemma}

\begin{lemma}\protect{\cite[Proposition 5.3.4]{L14}}\label{lem:translate-to-invertibility}
Let $A\in M_2(\mathbb H)$ be non-invertible. Then there exists $\lambda\in\mathbb H$ such that $\lambda$ is not a left eigenvalue of $A$ (i.e., $Av=\lambda v$ for some $v\in \mathbb H^2$) and
$A-\lambda I_2$ is invertible.
\end{lemma}

\section{Images of the power map with constant}\label{sec:image}

In this section, we prove Theorem~\ref{thm:image-phi}. The proof splits into two parts, the
first part is uniform across $\mathbb D\in\{\mathbb C,\mathbb H\}$ and any
positive integer $k$; the same argument
also covers $\mathbb D=\mathbb R$ when $k$ is odd, since in that case every
real number admits a unique real $k$-th root. The second part gives the proof of
remaining case $\mathbb D=\mathbb R$ with $k$ even, where the existence of
real $k$-th roots in $M_2(\mathbb R)$ may fail and a finer analysis is
required.

\begin{proof}[Proof of Theorem \ref{thm:image-phi}:]
It is easy to see that the row dependency ratio is preserved under right multiplication by an
arbitrary matrix in $M_2(\mathbb D)$, that is, if $A$ has the first form
in \eqref{eq:intro-ratio-form} and $X\in M_2(\mathbb K)$ is arbitrary,
then $AX$ has the same first form in \eqref{eq:intro-ratio-form} with the same parameter $t$. The case $\rho(A)=\infty$ is analogous. By iterating, we see that $AX^k$ has
row dependency ratio equal to $\rho(A)$ whenever 
$AX^k\neq 0$. This shows that ${\rm Im}(\phi_{A,k})\subset \{\,M\in M_2(\mathbb{D}) : \rho(M)=\rho(A)\,\}\cup\{\theta\}.$
For the other inclusion, we split the argument according to the underlying field and the parity
of $k$.

\medskip
\noindent\textbf{Part I:} \emph{The cases $\mathbb D\in\{\mathbb C,\mathbb H\}$ with $k$ both even and odd,
and $\mathbb D=\mathbb R$ with $k$ odd.}

Fix $C=\begin{pmatrix}
    p&q\\tp &tq
\end{pmatrix}$ and $A=\begin{pmatrix}
    a_1&a_2\\ta_1&ta_2
\end{pmatrix}$ with $p,q, a_1,a_2\in\mathbb D$.
We exhibit $X\in M_2(\mathbb D)$ such that $AX^k=C$, distinguishing three
cases according to which of $a_1,a_2$ vanish.

\medskip
\noindent\textit{Case 1: $a_1\neq 0$ and $a_2\neq 0$.}\;
By Lemma~\ref{lem:roots-H} (resp.\ its complex analogue, and the fact
that every real number admits a unique real $k$-th root when $k$ is
odd), the elements $a_1^{-1}p$ and $a_2^{-1}q$ admit $k$-th roots in
$\mathbb D$. Choose
$ X=\begin{pmatrix}
      (a_1^{-1}p)^{1/k} & 0\\
      0 & (a_2^{-1}q)^{1/k}
   \end{pmatrix},$
   so that
   $X^k=\begin{pmatrix}
      a_1^{-1}p & 0\\
      0 & a_2^{-1}q
   \end{pmatrix}.$
A direct computation gives
\[
   AX^k=
   \begin{pmatrix} a_1 & a_2\\ ta_1 & ta_2\end{pmatrix}
   \begin{pmatrix} a_1^{-1}p & 0\\ 0 & a_2^{-1}q\end{pmatrix}
   =
   \begin{pmatrix} p & q\\ tp & tq\end{pmatrix}
   = C.
\]

\medskip
\noindent\textit{Case 2: $a_1\neq 0$, $a_2=0$.}\;
Suppose first that $q\neq 0$. Set $M \;=\; \begin{pmatrix} a_1^{-1}p & a_1^{-1}q\\ 1 & 0\end{pmatrix}.$
Then $\rm Det(M)=\left[a_1^{-1}q\right]\neq 0$ in the $\mathbb H$ case, and the
classical determinant yields the same conclusion over $\mathbb C$ and
$\mathbb R$; hence $M$ is invertible. Since $M$ is invertible, it
admits a $k$-th root $X\in M_2(\mathbb D)$ using Lemma~\ref{lem:image-power-map-real} and Lemma~\ref{lem:invertible-kth-root}, and
\[
   AX^k=
   \begin{pmatrix} a_1 & 0\\ ta_1 & 0\end{pmatrix}
   \begin{pmatrix} a_1^{-1}p & a_1^{-1}q\\ 1 & 0\end{pmatrix}
   =
   \begin{pmatrix} p & q\\ tp & tq\end{pmatrix}
   =C.
\]
If $q=0$, then $C=\begin{pmatrix} p & 0\\ tp & 0\end{pmatrix}$,
and we see that for $X=\begin{pmatrix} (a_1^{-1}p)^{1/k} & 0\\ 0 & 0\end{pmatrix}$
we have $AX^k=C$.

\medskip
\noindent\textit{Case 3: $a_1=0$, $a_2\neq 0$.}\;
This case is symmetric to Case~2. If $p\neq 0$, set
$M \;=\; \begin{pmatrix} 0 & 1\\ a_2^{-1}p & a_2^{-1}q\end{pmatrix},$
which is invertible since $\rm Det(M)=\left[a_2^{-1}p\right]\neq 0$; by Lemma~\ref{lem:image-power-map-real} and Lemma~\ref{lem:invertible-kth-root} there exists
$X\in M_2(\mathbb K)$ with $X^k=M$, and a direct computation shows that
$AX^k=C$. If $p=0$, take
$X \;=\; \begin{pmatrix} 0 & 0\\ 0 & (a_2^{-1}q)^{1/k}\end{pmatrix}.$
This exhaustively covers all cases and establishes the reverse inclusion
when $A$ has the first form. The argument when $A$ has the second form
in~\eqref{eq:intro-ratio-form} (i.e.\ $\rho(A)=\infty$) proceeds along
identical lines and is omitted.

\medskip\noindent
\textbf{Part II:} \emph{The case $\mathbb K=\mathbb R$ with $k$ even.}\\
Our aim is to find $X$ such that $AX^k=C$. In view of Lemma~\ref{lem:image-power-map-real}, it is enough to find a matrix $N$ with $Tr(N)\geq 0$, and ${\rm det}(N)\geq 0$, but not simultaneously zero, such that $AN=C$. Indeed, $N$ has $k$-th root, and that serves the purpose for $X$.
Suppose that $X^k=\begin{pmatrix}
    x_1&x_2\\x_3&x_4
\end{pmatrix}$, $A=\begin{pmatrix}
    a_1&a_2\\ta_1&ta_2
\end{pmatrix}$ and $C=\begin{pmatrix}
    p&q\\tp &tq
\end{pmatrix}$ in $AX^k=C$, we get the equations,
\begin{equation}\label{eq:linear-constraints}
   x_1 a_1 + x_3a_2 \;=\; p,
   \qquad
   x_2 a_1 + x_4 a_2 \;=\; q
\end{equation}
We write
$\mathbf{w}=(p, q)^T,
   \mathbf{v}=(a_1, a_2)^T,$
   are column vectors in $\mathbb{R}^2$, and let
$r^2 = a_1^2 + a_2^2 > 0$ when $\mathbf{v} \neq 0$, since $A$ is a non-zero matrix. Then we get
$M\mathbf{v}=\mathbf{w}$ where 
   $M=
   \begin{pmatrix} x_1 & x_3 \\ x_2 & x_4 \end{pmatrix}.$
   Note that the trace and the determinant of $X^k$ and $M$ are the same. Therefore, if we find $M$ satisfying the trace-determinant condition given in Lemma \ref{lem:image-power-map-real}(2b), we are done.

To determine $M$, we take $\mathbf{u}=(-a_2, a_1)^T$ so that $\{\mathbf{v},\mathbf{u}\}$ forms a basis for $\mathbb R^2$ and write $M\mathbf{u}=\mathbf{z}$ for an arbitrary
vector $\mathbf{z}=(z_1,z_2)^T \in \mathbb{R}^2$.
Let $V = (\mathbf{v}\ \mathbf{u})$, which satisfies
$\det V = r^2 > 0$, and we get that
\[
   M =(\mathbf{w}\ \mathbf{z}) V^{-1}
       =\frac{1}{r^2}
       \bigl( a_1 \mathbf{w} - a_2 \mathbf{z},\ a_2 \mathbf{w} + a_1 \mathbf{z} \bigr).
\]
Hence $\operatorname{tr} M=x_1 + x_4
   =
   \frac{\mathbf{w}\cdot\mathbf{v} + \mathbf{z}\cdot\mathbf{u}}{r^2},$ and 
   $\det M = \frac{p z_2 - q z_1}{r^2}r^2.$
The problem is now reduced to choosing $\mathbf{z} \in \mathbb{R}^2$ such
that $T:=\; \mathbf{w}\cdot\mathbf{v} + \mathbf{z}\cdot\mathbf{u} \;\geq\; 0,$ and 
$D:= p z_2 - 
   q z_1 \;\geq\; 0,$
with $T$ and $D$ not simultaneously zero.

We first assume that $\mathbf{v}$ is not collinear with $\mathbf{w}$. Since $\{\mathbf{u}, \mathbf{v}\}$ forms a basis for $\mathbb R^2$, we have $\mathbf{w}\cdot\mathbf{u} \neq 0$. Now, choose $\mathbf{z} = \eta \mathbf{u}
+ \zeta \mathbf{w}$ for parameters $\eta,\zeta \in \mathbb{R}$. A short
computation gives
$D=\eta (\mathbf{w}\cdot\mathbf{v}),$ and
$\mathbf{z}\cdot\mathbf{u} \;=\; \eta r^2 + \zeta (\mathbf{w}\cdot\mathbf{u}).$
Choose $\eta$ so that $D \geq 0$ and then $\zeta$ so that $T > 0$; both
choices are possible because $\mathbf{w}\cdot\mathbf{u} \neq 0$.
Now, assume that $\mathbf{v}$ is collinear with $\mathbf{w}$, that is,
$\mathbf{w} = \lambda \mathbf{v}$ for some $\lambda \in \mathbb{R}$,
and by Lemma \ref{lem:image-power-map-real}(3) there exists $X\in M_2(\mathbb R)$ such that $X^k = \lambda I$, and this implies that 
$A X^k=\lambda A=
   \begin{pmatrix} p & q \\ t p & t q \end{pmatrix}=C,$ as required.
Now, suppose $a_2 = 0$, then $A = \begin{pmatrix} a_1 & 0 \\ t a_1 & 0 \end{pmatrix}$, and the
constraint reduces to $x_1 a_1= p,$ and $x_2 a_1=q.$
Further, if $q \neq 0$, then it implies that $x_2\neq0$. Choose $x_4$ to make $x_1 + x_4 > 0$ and then $x_3$ to
      make $x_1x_4-x_2x_3 \geq 0$. Then $M$ satisfies the
      trace--determinant condition from Lemma~\ref{lem:image-power-map-real}(2).
 If $x_2 = 0$, by Lemma \ref{lem:image-power-map-real}(3) there exists $X\in M_2(\mathbb R)$ such that $X^k = x_1I$, and we have $AX^k=C$.
The case $a_1 = 0$ can be done analogously, and hence we omit the details. 
\end{proof}

\section{Surjectivity of the polynomial map with constant}\label{sec:Surjectivity}
In this section, we prove Theorem \ref{thm-surjectivity}. To prove Theorem \ref{thm-surjectivity}, we will use Theorem \ref{thm:image-phi}. We split this proof into Proposition \ref{prop-both-non-invertible}, Proposition \ref{thm:mixed-case}, and Proposition \ref{thm:doubly-invertible}. 

\begin{proposition}\label{prop-both-non-invertible}
Let $\mathbb D\in\{\mathbb R,\mathbb H\}$, and let $A,B\in M_2(\mathbb D)$ be non-zero non-invertible
matrices with row dependency ratios $t,s\in\mathbb D\cup\{\infty\}$, respectively.
Then the polynomial map $\omega(X,Y)=AX^{k_1}+BY^{k_2}$ is surjective if
and only if $t\neq s$.
\end{proposition}
\begin{proof}
Since both $A$ and $B$ are non-zero, non-invertible, we assume that $A$ and $B$ both take the form
in~\eqref{eq:intro-ratio-form}, with row dependency ratios
$t,s\in\mathbb D$. In view of Theorem~\ref{thm:image-phi}, we write that $A X^{k_1}=
   \begin{pmatrix} x_1 & x_2 \\ tx_1 & tx_2 \end{pmatrix}$ and $B Y^{k_2}=
   \begin{pmatrix} y_1 & y_2 \\ sy_1 & sy_2 \end{pmatrix},$
for some $x_1,x_2,y_1,y_2\in\mathbb D$.

\smallskip
\noindent\textbf{Case 1: $t\neq s$.} Let
$C=\begin{pmatrix} \ell & m \\ n & o \end{pmatrix}\in M_2(\mathbb D)$.
Now $\omega(X,Y)=C$ gives rise to the following system of equations.
\[
   x_1 + y_1 = \ell,\quad
   x_2 + y_2 = m,\quad
   tx_1 + sy_1 = n,\quad
   tx_2 + sy_2 = o.
\]
Since $t\neq s$, the system has a unique
solution, given by
\[
   x_1 = (t-s)^{-1}(n-s\ell),\quad
   x_2 = (t-s)^{-1}(o-sm),
\]
\[
   y_1 = (t-s)^{-1}(t\ell-n),\quad
   y_2 = (t-s)^{-1}(tm-o).
\]
Hence,  $\omega$ is surjective by Theorem~\ref{thm:image-phi}.

\smallskip
\noindent\textbf{Case 2: $t=s$.} Then for every $X,Y\in M_2(\mathbb D)$ we have
\[
   \omega(X,Y)
   \;=\;
   \begin{pmatrix}
      x_1+y_1 & x_2+y_2 \\
      t(x_1+y_1) & t(x_2+y_2)
   \end{pmatrix},
\]
so every element of ${\rm Im}(\omega)$ has row dependency ratio $t$. In
particular, the identity matrix $I_2$ does not lie in ${\rm Im}(\omega)$, so
$\omega$ is not surjective.

 If both $A$ and
$B$ have second form in~\eqref{eq:intro-ratio-form} (so $t=s=\infty$),
then by Theorem~\ref{thm:image-phi} every element of ${\rm Im}(\omega)$ has
zero first row, and again $I_2\notin{\rm Im}(\omega)$. If exactly one of
$A,B$ has the second form in~\eqref{eq:intro-ratio-form} (say $B$, so
$s=\infty$ and $t\in\mathbb D$), then a direct adaptation of Case~1 above shows
that $\omega$ is surjective. Indeed, given
$C=\begin{pmatrix} \ell & m \\ n & o \end{pmatrix}$, we choose
$X^{k_1}$ so that $AX^{k_1}=\begin{pmatrix} \ell & m \\ t\ell & tm \end{pmatrix}$
(possible by Theorem~\ref{thm:image-phi}) and then $BY^{k_2}=C-AX^{k_1}$,
which has zero first row and is therefore in ${\rm Im}(\phi_{B,k_2})$.
\end{proof}

\begin{proposition}\label{thm:mixed-case}
Let $\mathbb D\in\{\mathbb R,\mathbb H\}$, and let $A,B\in M_2(\mathbb D)$ be non-zero matrices with exactly one of them invertible. Then for all integers $k_1,k_2\ge 1$ the polynomial map with constants $\omega(X,Y) \;=\; AX^{k_1}+BY^{k_2}$
is surjective on $M_2(\mathbb D)$.
\end{proposition}

\begin{proof}
Without loss of generality, we assume that $A$ is invertible and $B$ is non-zero, non-invertible. Multiplying the equation $\omega(X,Y)=C$ on the left by $A^{-1}$, so it suffices to prove that the equation
\begin{equation}\label{eq:reduced-mixed}
   X^{k_1}+B'Y^{k_2} \;=\; C
\end{equation}
admits a solution for every $C\in M_2(\mathbb D)$, where $B'=A^{-1}B$ is again non-zero non-invertible. We rename $B'$ to $B$ to lighten the notation.

\medskip
\noindent\textbf{Part I:} \emph{The cases $\mathbb D=\mathbb H$.}

\smallskip
\noindent\textit{Case 1: $C$ is invertible.} By Lemma~\ref{lem:invertible-kth-root} there exists $X\in M_2(\mathbb D)$ with $X^{k_1}=C$. Taking $Y=0$ solves~\eqref{eq:reduced-mixed}.

\smallskip
\noindent\textit{Case 2: $C$ is non-invertible.} By Lemma~\ref{lem:non-invertible-form}, $C$ has the form $\begin{pmatrix} p & q \\ sp & sq\end{pmatrix}$ for some $s\in\mathbb D$ (the case where the first row of $C$ is zero is treated symmetrically); the case $p=q=0$ is trivial. Without loss of generality $q\neq 0$. By Theorem~\ref{thm:image-phi}, $BY^{k_2}$ takes the form
$BY^{k_2} \;=\; \begin{pmatrix} \beta_1 & \beta_2 \\ t\beta_1 & t\beta_2\end{pmatrix}$
for $\beta_1,\beta_2\in\mathbb D$, and $t$ is the row dependency ratio of $B$. We split into sub-cases.

\smallskip
\noindent\textit{Sub-case 2(a): $t=s$.} Set $\beta_1=p$, $\beta_2=q$. Then $C-BY^{k_2}=0$, and $X=0$ solves~\eqref{eq:reduced-mixed}.

\smallskip
\noindent\textit{Sub-case 2(b): $t\neq s$, $p\neq0$.} Set $\beta_1=p$, $\beta_2=q/2$. Then
\[
   C-BY^{k_2} \;=\;
   \begin{pmatrix} 0 & q/2 \\ (s-t)p & (s-t/2)q\end{pmatrix}.
\]
The Dieudonn\'e determinant of this matrix is $\left[q/2(s-t)p\right]\neq 0$, so the matrix is invertible, and Lemma \ref{lem:invertible-kth-root} yields $X\in M_2(\mathbb D)$ with $X^{k_1}=C-BY^{k_2}$. 

\smallskip
\noindent\textit{Sub-case 2(c): $t\neq s$, $p=0$.} Set $\beta_1=-1$, $\beta_2=q$. Then
\[
   C-BY^{k_2} \;=\;
   \begin{pmatrix} 1 & 0 \\ t & (s-t)q\end{pmatrix}.
\]
The Dieudonn\'e determinant of this matrix is $\left[(s-t)q\right]\neq 0$, so the matrix is invertible, and Lemma \ref{lem:invertible-kth-root} yields $X\in M_2(\mathbb D)$ with $X^{k_1}=C-BY^{k_2}$. 

The case when the first row of $C$ is zero, and the case where the row dependency ratio of $B$ equals $\infty$, can be done in a similar way, and hence we omit the details.

\medskip
\noindent\textbf{Part II:} \emph{The case $\mathbb D=\mathbb R$.}

Since $B$ is non zero, non invertible, by Theorem~\ref{thm:image-phi}, we can write $BY^{k_2}=
   \begin{pmatrix} \beta_1 & \beta_2 \\ t\beta_1 & t\beta_2\end{pmatrix},$
for $\beta_1,\beta_2\in\mathbb R$. Set $X^{k_1}=\begin{pmatrix} x_1 & x_2 \\ x_3 & x_4\end{pmatrix}$ and let $C=\begin{pmatrix} p & q \\ r & s\end{pmatrix}\in M_2(\mathbb R)$ be arbitrary. Then $X^{k_1}+BY^{k_2}=C$ gives
\begin{equation}\label{eq:linear-constraint-mixed-R_0}
   x_1+\beta_1=p,\qquad x_2+\beta_2=q,\qquad x_3+t\beta_1=r,\qquad x_4+t\beta_2=s.
\end{equation}
Eliminating $\beta_1$ and $\beta_2$, we have 
\begin{equation}\label{eq:linear-constraint-mixed-R}
   tx_1-x_3 \;=\; tp-r,\qquad tx_2-x_4 \;=\; tq-s.
\end{equation}
Suppose that $k_1$ is even.
We now proceed as in the proof of Theorem~\ref{thm:image-phi}(II), in \eqref{eq:linear-constraints} with the substitutions
\[
   (a_1,a_2) \;\longleftrightarrow\; (t,-1),
   \qquad
   (p,q) \;\longleftrightarrow\; (tp-r,\,tq-s).
\]
Since the substituted vector $(t,-1)$ is non-zero, there exists $X^{k_1}$ satisfying~\eqref{eq:linear-constraint-mixed-R} with $\operatorname{tr}(X^{k_1})>0$ and $\det(X^{k_1})\ge 0$. By Lemma \ref{lem:image-power-map-real}(2), such a matrix lies in $X^{k_1}\in\operatorname{Im}(P_{k_1})$, and hence we are done. The proof is similar when $k_1$ is odd. Once we get $X^{k_1}$, the parameters $\beta_1,\beta_2$ are determined by the equations \eqref{eq:linear-constraint-mixed-R}, and hence in view of Theorem~\ref{thm:image-phi}, there exists $Y\in M_2(\mathbb R)$ such that $BY^{k_2}=
   \begin{pmatrix} \beta_1 & \beta_2 \\ t\beta_1 & t\beta_2\end{pmatrix}.$ Now, it is clear that for the choice of $X, Y\in M_2(\mathbb R)$, obtained above, we get $X^{k_1}+BY^{k_2}=C$. This proves Part II.
\end{proof}

\begin{proposition}\label{thm:doubly-invertible}
Let $\mathbb D\in\{\mathbb R, \mathbb H\}$, and let $A,B\in M_2(\mathbb D)$ both be invertible. Then for all integers $k_1,k_2\ge 1$ the polynomial map with constants
$\omega(X,Y) \;=\; AX^{k_1}+BY^{k_2}$
is surjective on $M_2(\mathbb D)$.
\end{proposition}

\begin{proof}
Multiplying $\omega(X,Y)=C$ on the left by $A^{-1}$ reduces the problem to
\begin{equation}\label{eq:reduced-doubly}
   X^{k_1}+B'Y^{k_2} \;=\; C,\qquad C\in M_2(\mathbb D),
\end{equation}
where $B'=A^{-1}B\in M_2(\mathbb D)$ is again invertible. We rename $B'$ to $B$.

\medskip
\noindent\textbf{Part I:} \emph{The cases $\mathbb D=\mathbb H$.}

By Lemma~\ref{lem:jordan-H}, choose an invertible $P\in M_2(\mathbb D)$ such that $J:=P^{-1}BP$ takes one of the three canonical forms
\[
   \lambda I_2\ (\lambda\neq 0),
   \qquad
   \begin{pmatrix} \lambda & 0\\ 0 & \mu\end{pmatrix}\ (\lambda,\mu\neq 0,\ \lambda\neq\mu),
   \qquad
   \begin{pmatrix} \lambda & 1\\ 0 & \lambda\end{pmatrix}\ (\lambda\neq 0).
\]
Conjugating~\eqref{eq:reduced-doubly} by $P$ and renaming variables, it suffices to solve
\begin{equation}\label{eq:JordanForm-doubly}
   X^{k_1}+JY^{k_2} \;=\; C
\end{equation}
for arbitrary $C\in M_2(\mathbb D)$, with $J$ one of the three forms above. 

\smallskip
\noindent\textit{Case I.1: $J=\lambda I_2$ with $\lambda\neq 0$.} The equation becomes $X^{k_1}+\lambda Y^{k_2}=C$, which is surjective on $M_2(\mathbb D)$ by \cite[Theorem 5.2]{SPA25}.

\smallskip
\noindent\textit{Case I.2: $J=\begin{pmatrix}\lambda & 0\\ 0 & \mu\end{pmatrix}$ with $\lambda,\mu\neq 0$ and $\lambda\neq\mu$.} If $C$ is invertible, take $Y=0$ and apply Lemma~\ref{lem:invertible-kth-root} to find $X$ with $X^{k_1}=C$. If $C$ is non-invertible, then by Lemma~\ref{lem:translate-to-invertibility} there exists $a\in\mathbb D$ such that $C-aI_2$ is invertible. Choose $\alpha,\beta\in\mathbb D$ with
$\lambda\alpha=\mu\beta= a,$
and set $Y^{k_2}=\begin{pmatrix}\alpha & 0\\ 0 & \beta\end{pmatrix}$. A direct computation gives
$JY^{k_2} \;=\;
   \begin{pmatrix}\lambda\alpha & 0\\ 0 & \mu\beta\end{pmatrix}
   =aI_2.$
Hence $X^{k_1}=C-aI_2$ is invertible, and Lemma \ref{lem:invertible-kth-root} produces $X$.

\smallskip
\noindent\textit{Case I.3: $J=\begin{pmatrix}\lambda & 1\\ 0 & \lambda\end{pmatrix}$.} Since $J$ is invertible, $\lambda\neq 0$. If $C$ is invertible, take $Y=0$ and apply Lemma~\ref{lem:invertible-kth-root}. If $C$ is non-invertible, choose $a\in\mathbb D$ with $a\neq 0$ such that $C-aI_2$ is invertible (Lemma~\ref{lem:translate-to-invertibility}). Set $\lambda\alpha=a$, and define the invertible matrix
$Y^{k_2} \;=\;
   \begin{pmatrix}\alpha & -\lambda^{-1}\alpha\\ 0 & \alpha\end{pmatrix}.$
A direct computation gives
\[JY^{k_2} \;=\;
   \begin{pmatrix}\lambda & 1\\ 0 & \lambda\end{pmatrix}
   \begin{pmatrix}\alpha & -\lambda^{-1}\alpha\\ 0 & \alpha\end{pmatrix}
   \;=\;
   \begin{pmatrix}\lambda\alpha & 0\\ 0 & \lambda\alpha\end{pmatrix}
   \;=\; aI_2,
\]
so $X^{k_1}=C-aI_2$ is invertible, and the conclusion follows from Lemma~\ref{lem:invertible-kth-root}.

\medskip
\noindent\textbf{Part II:} \emph{The case $\mathbb D=\mathbb R$.}

Let $D=P^{-1}BP$, where $P\in\GL_2(\mathbb R)$. Then  $D$ falls into four mutually exclusive and exhaustive cases:
\begin{itemize}
   \item[\textbf{(D1)}] $D=\lambda I$ with $\lambda\neq 0$;
   \item[\textbf{(D2)}] $D=\diag(\lambda,\mu)$ with $\lambda\neq\mu$ and $\lambda\mu\neq 0$;
   \item[\textbf{(D3)}] $D=\begin{pmatrix}\lambda & 1\\ 0 & \lambda\end{pmatrix}$ with $\lambda\neq 0$;
   \item[\textbf{(D4)}] $D=\begin{pmatrix} a & -b\\ b & a\end{pmatrix}$ with $b\neq 0$.
\end{itemize}

Conjugating  of~\eqref{eq:reduced-doubly} by $P$, the problem is equivalent to solving $X^{k_1}+DY^{k_2} \;=\; C$
for arbitrary $C=\begin{pmatrix} p & q \\ r & s\end{pmatrix}\in M_2(\mathbb R)$. We now consider each of the above cases separately. In every case constructed below, the positive real numbers $M$ and $N$ are free parameters, which may be chosen arbitrarily large.

\smallskip
\noindent\textit{Case (D1):} Set
$X^{k_1} \;=\;
   \begin{pmatrix} M+p & q\\ r & M+s\end{pmatrix},$ and $
   Y^{k_2} \;=\;
   \begin{pmatrix} -M/\lambda & 0\\ 0 & -M/\lambda\end{pmatrix}.$
Then $DY^{k_2}=-MI$, so
$X^{k_1}+DY^{k_2}=
   \begin{pmatrix} p & q\\ r & s\end{pmatrix}=C.$
For $M$ sufficiently large,
$\operatorname{tr}(X^{k_1})=2M+p+s>0,$ and $\det(X^{k_1})= (M+p)(M+s)-qr>0,$
so by Lemma~\ref{lem:image-power-map-real} the matrix $X^{k_1}$ lies in $\operatorname{Im}(P_{k_1})$. The matrix $Y^{k_2}$ is scalar and hence in $\operatorname{Im}(P_{k_2})$.

\smallskip
\noindent\textit{Case (D2):} We split this case into two sub-cases according to the signs of the eigenvalues.

\noindent\emph{Sub-case (D2.a): $\lambda$ and $\mu$ have the same sign.} Choose
$X^{k_1}=
   \begin{pmatrix} M+p & q\\ r & \tfrac{\mu}{\lambda}M+s\end{pmatrix},$ and $Y^{k_2}=
   \begin{pmatrix} -M/\lambda & 0\\ 0 & -M/\lambda\end{pmatrix}.$
Then $DY^{k_2}=
   \begin{pmatrix} -M & 0\\ 0 & -\tfrac{\mu}{\lambda}M\end{pmatrix},$
and hence $X^{k_1}+DY^{k_2}=C$. For large $M$, both $\operatorname{tr}(X^{k_1})$ and $\det(X^{k_1})$ are positive (since $\mu/\lambda>0$), so by Lemma~\ref{lem:image-power-map-real} $X^{k_1}$ lies in $\operatorname{Im}(P_{k_1})$. The matrix $Y^{k_2}$ is scalar and hence in $\operatorname{Im}(P_{k_2})$.

\noindent\emph{Sub-case (D2.b): $\lambda$ and $\mu$ have opposite signs.} Without loss of generality $\lambda>0$ and $\mu<0$. The construction depends on the entry pattern of $C$. In each sub-case below, $M,N>0$ are taken sufficiently large that the trace and determinant of the constructed matrix $X^{k_1}$ are positive.
\begin{itemize}
\item[(i)] Suppose that $q>0$. Take $X^{k_1}=
   \begin{pmatrix} p & q/2\\ r+N\mu & s-M\mu\end{pmatrix}$ and $Y^{k_2}=
   \begin{pmatrix} 0 & q/(2\lambda)\\ -N & M\end{pmatrix}.$
Then $DY^{k_2}=\begin{pmatrix} 0 & q/2\\ -N\mu & M\mu\end{pmatrix}$, and so $X^{k_1}+DY^{k_2}=C$.

\item[(ii)] Suppose that $q<0$. Take $X^{k_1}=
   \begin{pmatrix} p & q/2\\ r-N\mu & s-M\mu\end{pmatrix}$ and $Y^{k_2}=
   \begin{pmatrix} 0 & q/(2\lambda)\\ N & M\end{pmatrix},$ then $X^{k_1}+DY^{k_2}=C$

\item[(iii)] Suppose that $q=0$ and $r>0$. Take $X^{k_1}=
   \begin{pmatrix} p & -N\lambda\\ r/2 & s-M\mu\end{pmatrix},$ and $Y^{k_2}=
   \begin{pmatrix} 0 & N\\ r/(2\mu) & M\end{pmatrix},$ then $X^{k_1}+DY^{k_2}=C.$

\item[(iv)] Suppose that $q=0$ and $r<0$. We take
$X^{k_1}=
   \begin{pmatrix} p & N\lambda\\ r/2 & s-M\mu\end{pmatrix},$ and $Y^{k_2}=
   \begin{pmatrix} 0 & -N\\ r/(2\mu) & M\end{pmatrix},$ then $X^{k_1}+DY^{k_2}=C.$
\item[(v)] Suppose that $q=r=0$ and $p>s$. Take $X^{k_1} =sI,$ and $Y^{k_2}=
   \begin{pmatrix} (p-s)/\lambda & 0\\ 0 & 0\end{pmatrix},$ then $X^{k_1}+DY^{k_2}=C.$

\item[(vi)] Suppose that $q=r=0$ and $s>p$. Take
$X^{k_1} \;=\; \frac{p\mu-\lambda s}{\mu-\lambda}\,I,$ and $
   Y^{k_2} \;=\; \frac{s-p}{\mu-\lambda}\,I.$
A direct computation verifies $X^{k_1}+DY^{k_2}=C$.

\item[(vii)] $q=r=0$ and $p=s$. Take $X^{k_1}=pI$ and $Y^{k_2}=0$. Then trivially $X^{k_1}+DY^{k_2}=pI=C$.
\end{itemize}
In each sub-case a routine computation verifies $X^{k_1}+DY^{k_2}=C$, and the choices of $M$ and $N$ ensure $X^{k_1},Y^{k_2}\in\operatorname{Im}(P_{k_i})$ by Lemma \ref{lem:image-power-map-real}. 

\smallskip
\noindent\textit{Case (D3):} Set
$X^{k_1}=
   \begin{pmatrix} M+p & q+M/\lambda\\ r & s+M\end{pmatrix},$ and $
   Y^{k_2}=
   \begin{pmatrix} -M/\lambda & 0\\ 0 & -M/\lambda\end{pmatrix}.$
A direct computation gives
$DY^{k_2}=
   \begin{pmatrix} -M & -M/\lambda\\ 0 & -M\end{pmatrix},$
so $X^{k_1}+DY^{k_2}=C$. For sufficiently large $M$, $X^{k_1},Y^{k_2}\in\operatorname{Im}(P_{k_i})$ by Lemma \ref{lem:image-power-map-real}.

\smallskip
\noindent\textit{Case (D4):} Set
$Y^{k_2}=D^{-1}(-MI),$ and $X^{k_1}=
   \begin{pmatrix} M+p & q\\ r & M+s\end{pmatrix}.$
Then $DY^{k_2}=-MI$ and $X^{k_1}+DY^{k_2}=C$. For sufficiently large $M$, by Lemma \ref{lem:image-power-map-real} we see that  $X^{k_1}, Y^{k_2}\in \operatorname{Im}(P_{k_2}).$ 
This completes the proof.
\end{proof}

\bibliography{References-2-2-2.bib}
\bibliographystyle{amsalpha}

\end{document}